\RequirePackage[l2tabu,orthodox]{nag}
\documentclass[11pt,DIV=calc]{scrartcl}

\usepackage[paperwidth=165mm, paperheight=242mm, left=22.5mm, right=22.5mm, top=20mm, bottom=20mm]{geometry}

\usepackage[utf8]{inputenc}
\usepackage[T1]{fontenc}
\usepackage{lmodern}

\usepackage{scrlayer-scrpage}

\usepackage[shortlabels]{enumitem}

\usepackage{booktabs}

\usepackage{array}

\usepackage[usenames,dvipsnames,svgnames,table]{xcolor}

\usepackage{subcaption}

\usepackage[all,warning]{onlyamsmath}
\usepackage{mathtools}
\usepackage{amssymb}
\usepackage{amsthm}

\usepackage{biblatex}
\bibliography{short}

\usepackage{graphicx}

\usepackage{tikz}
\usetikzlibrary{babel,positioning,external}
\usepackage{float}

\DeclareMathOperator{\prob}{\mathbb{P}}

\DeclareMathOperator{\expect}{\mathbb{E}}

\newcommand{\R}{\ensuremath{\mathbb{R}}}

\newtheorem{theorem}{Theorem}
\newtheorem{lemma}{Lemma}
\newtheorem{proposition}{Proposition}
\newtheorem{corollary}{Corollary}
\theoremstyle{definition}

\DeclareMathOperator{\pbinom}{Bin}
\DeclareMathOperator{\ppois}{Po}
\DeclareMathOperator{\pbeta}{Beta}

\DeclareMathOperator{\pnorm}{N}

\usepackage{amsfonts}
\DeclareMathSymbol{\shortminus}{\mathbin}{AMSa}{"39}

\DeclareSymbolFont{ttfont}{OT1}{cmtt}{m}{n}
\DeclareMathSymbol{\splus}{\mathord}{ttfont}{`+}
\DeclareMathSymbol{\sminus}{\mathord}{ttfont}{`-}
\DeclareMathSymbol{\snull}{\mathord}{ttfont}{`0}

\recalctypearea

\begin{document}

\title{Another look at binomial and related distributions exceeding values close to their centre}

\author{Tilo Wiklund}
\date{\small Department of Mathematics, Uppsala University, SE-751 06 Uppsala, Sweden}

\maketitle

\begin{abstract}
  We generalise the known fact that for binomial \(X_{n,k} \sim \pbinom(n,
  k/n)\) one has \(\inf_{k>1,n} \prob(X_{n,k} \geq k) \geq \lim_{k \to
    1+}\prob(X_{2,k} \geq k) = 1/4\) to cover probabilities of exceeding a
  constant shift away from the mean. The proof is very short and the theorem
  yields the original result as a special case, as well as proving that an analogous
  result holds for the Poisson distribution. We furthermore prove a similar
  property holds for the family of beta distributions near their mode. Thanks to
  the connection between Binomial and Beta distributions, this allows us to shed
  some further light on the original result regarding Binomial probabilities.
\end{abstract}

\section{Introduction}

For binomial random variables \(X_{n,\mu} \sim \pbinom(n, \mu/n)\) with
expectations \(\expect(X_{n,\mu}) = \mu \geq 1\) it was proved by
\citeauthor{Greenberg2014}\cite{Greenberg2014} that \(\prob(X_{n,\mu} \geq \mu)
\geq 1/4\). While a lot is known concerning the concentration behaviour of, say,
sums of independent random variables, inequalities of the kind above appear
somewhat more obscure. Nevertheless, they have recently garnered some interest.
For example, a more elementary proof of the above statement was given by
\citeauthor{Doerr2018}\cite{Doerr2018} and other results of similar type can be
found in Lemma~13 of \citeauthor{Rigollet2011}\cite{Rigollet2011} as well as the
main results of \citeauthor{Pelekis2016b}\cite{Pelekis2016b}. A closely related
question related to the binomials was solved by
\citeauthor{janson2021}\cite{janson2021} and
\citeauthor{barabesi2023}\cite{barabesi2023}, who dealt with the question of
which \(\mu\) yields the extremal probability for any given \(n\).

Inequalities of this type turn out to be useful in learning theory and the
analysis of algorithm. For details we refer to the references given by
\citeauthor{Doerr2018}. The author recently had to control such \emph{skewness}
relative to the expected value for a certain family of probability
measures\cite{Wiklund2018}. Though the quantity of interest was not an
expectation, the problem was, in spirit, as follows. Say that a function \(f\) is
concave on the support of some family \((X_{\theta})_{\theta \in \Theta}\) of
random variables. By Jensen's inequality we know \[\expect(f(X_{\theta})) \leq
  f(\expect(X_{\theta})).\] If \(f\) is also non-negative and non-decreasing then \[\expect(f(X_{\theta})) \geq
  \expect(f(X_{\theta})\mathbf{1}(X_{\theta} \geq \expect(X_{\theta}))) \geq
  f(\expect(X_{\theta}))\prob(X_{\theta} \geq \expect(X_{\theta})).\] This very
cheap lower bound is proportional to the upper bound from Jensen's inequality as
long as the family \((X_{\theta})_{\theta \in \Theta}\) never gets \emph{too
  (positively) skewed}. Specifically when \(\prob(X_{\theta} \geq m_{\theta})\)
stays away from \(0\) where \(m_{\theta} = \expect(X_{\theta})\). More generally,
one may ask this question about any \(m_{\theta}\) too much in the centre of the
distribution for any standard concentration inequalities to be applicable.

By observing that the result of \citeauthor{Greenberg2014} is a special case of
an inequality previously used by \citeauthor{Anderson1967}\cite{Anderson1967}
and \citeauthor{Jogdeo1968}\cite{Jogdeo1968}, resulting from a (less famous)
inequality of Hoeffding's\cite{Hoeffding1956}, we give a natural generalisation
of \citeauthor{Greenberg2014}'s result. This generalisation covers probabilities
\(\prob(X_{n,\mu} \geq \mu + l)\) and gives the sharp refinement where the lower
bound is allowed to depend on \(\mu\). The proofs are almost trivial once one
realises the applicability of Hoeffding's result.

The bound for \(l = 1\) was included in \citeauthor{Doerr2018}, but proved by
different means. While lower bounds for more general \(l > 0\) have also been
given by \citeauthor{Pelekis2016}\cite{Pelekis2016}, these are not directly
comparable. We identify the minimising (sequence of) parameters, yielding sharp
bounds when the lower bound is allowed to depend on (at most) \(\mu\) while the
bound of \citeauthor{Pelekis2016} depends on both \(n\) and \(\mu\) but is not
exact at the minimum. Similarly to \citeauthor{Pelekis2016b}\cite{Pelekis2016b}
this method yields analogous bounds also for the Poisson distribution.

An example of such results for a family of continuous distributions can be found
in \citeauthor{Arab2021convex}\cite{Arab2021convex}. There, results were given
for the family of Beta distribution when one parameter is increased and the
other is decreased. We prove here, by rather different means, monotonicity
results for the family of Beta distributions near their mode (when it exists),
when both parameters increase or both parameters decrease. As a
special case we find that for a fixed mode the probability of being on either
side of the mode tend monotonically to \(1/2\) as both parameters increase.

The proof concerning binomial probabilities almost says that quantities of the
type \(\prob(X_{n,\mu} \geq \mu + l)\) are monotone in \(n\). The failure to
actually be monotonic is in some sense an artefact of the discreteness of the
Binomial distribution. It is well known that probabilities relating to binomial
distributions can be formulated in terms of probabilities related to beta
distributions. This connection yields a smooth interpolation of
\(\prob(X_{n,\mu} \geq \mu + l)\) in terms of an analogous quantity for beta
distributions.

Numerical experiments suggest this quantity is actually monotone in \(n\). Using
our proof concerning monotonicity of beta probabilities near their mode actually
proves this fact, but only for certain ranges of \(\mu\) and \(l\). This gives
some additional clarity to the statement about binomial probabilities and yields
smooth bounds in terms of \(\mu/n\).

\section{Main Results}

Throughout this section we will have \(X_{n,\mu} \sim \pbinom(n, \mu/n)\) denote
binomial random variables with expected values \(\expect(X_{n,\mu}) = \mu\) and
\(Z_{\lambda} \sim \ppois(\lambda)\) Poisson random variables with expected
values \(\expect(Z_{\lambda}) = \lambda\). In other words \(\prob(X_{n,\mu} = k)
= n^{-n}\binom{n}{k}k^{\mu}(n - k)^{n - \mu}\) for \(k = 0, 1, \dotsc, n\) and
\(\prob(Z_{\lambda} = k) = \lambda^{k}e^{-\lambda}/k!\) for \(k = 0, 1,
\dotsc\). Our objective is to bound, from below, \(\prob(X_{n,\mu} \geq \mu +
l)\) for ``small'' \(l\). Examples of such probabilities are given in
Figure~\ref{fig:fig1}. We begin by making a number of observations.

For fixed \(p = \mu/n\) and \(n\) tending to infinity one has
\(\lim_{n\to\infty} \prob(X_{n,\mu} \geq \mu+l) = 1/2\) by the central limit
theorem while for \(\mu\) fixed with \(n\) tending to infinity one finds a
limiting Poisson probability \(\prob(Z_{\mu} \geq \mu+l)\). Similarly by the
central limit theorem \(\lim_{\lambda \to \infty}\prob(Z_{\lambda} \geq
\lambda+l) = 1/2\). As far as these asymptotics are concerned there is
therefore not much to say.

Finding a sharp lower-bound depending only on \(l\), independent of \(n\) and
\(\mu\), is equivalent to identifying minimising \(n\) and \(\mu\). A
non-trivial global lower bound of this type is not possible, since \(\lim_{\mu
  \to 0+} \prob(X_{n,\mu} \geq \mu) = \lim_{\mu \to 0+} \prob(X_{n,\mu} \neq 0)
= 0\) and certainly \(\prob(X_{n,n-l} \geq (n-l)+l+1) = \prob(X_{n,n-l} \geq
n+1) = 0\). In the result of \citeauthor{Greenberg2014} this is circumvented by
the constraint \(\mu \geq 1\). In light of the above observation, in generalising
their result, we will have to assume \(\mu \leq n-l\). For \(l = 0\), the
special case of \citeauthor{Greenberg2014}, this is of course no additional
constraint.

\begin{figure}[ht]
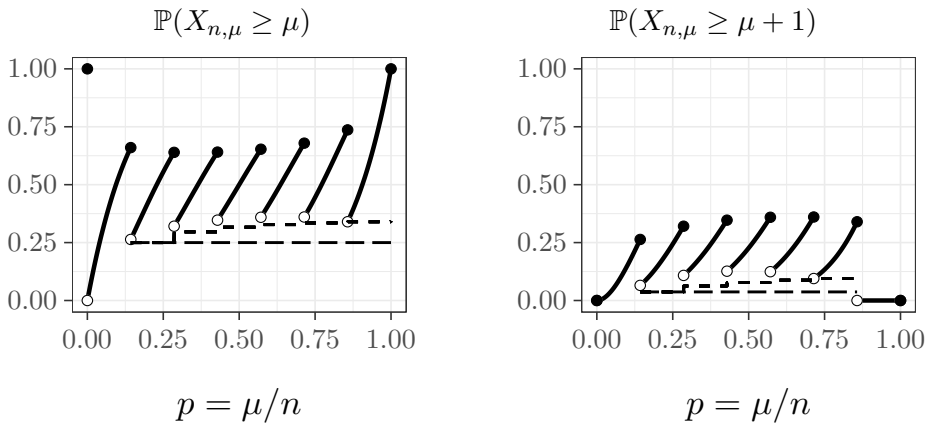

  \centering
  \begin{minipage}{0.46\textwidth}
    \centering{}
    \(\prob(X_{n,\mu} \geq \mu)\)
    \mbox{\hspace{-1.25cm}\input{pictures/plot1.tex}}
  \end{minipage}
  \begin{minipage}{0.46\textwidth}
    \centering{}
    \(\prob(X_{n,\mu} \geq \mu+1)\)
    \mbox{\hspace{-0.5cm}\input{pictures/plot2.tex}}
  \end{minipage}
  \caption{Probability of binomial random variable with \(n = 5\) trials being
    above (left) and at least \(1\) above (right) expected value \(k = pn\), with
    the bounds from Corollary~\ref{cor:binom2} marked by dashed lines}
  \label{fig:fig1}
\end{figure}

Since \(X_{n,\hspace{0.3mm}\mu}\) is supported on \(\{ 0, 1, \dotsc, n \}\) we
have \(\prob(X_{n,\hspace{0.3mm}\mu} \geq \mu+l) = \prob(X_{n,\hspace{0.3mm}\mu}
\geq \lceil \mu+l \rceil)\) and \(\prob(X_{n,\hspace{0.3mm}\mu} \leq \mu-l) =
\prob(X_{n,\hspace{0.3mm}\mu} \leq \lfloor \mu-l \rfloor)\), where \(\lceil x
\rceil\) and \(\lfloor x \rfloor\) denotes the least integer greater and
greatest integer less than \(x\). This gives rise to the jump discontinuities
at integer values of \(\mu + l\) seen in Figure~\ref{fig:fig1}. By a standard
coupling argument \(X_{n,\mu}\) is increasing in \(\mu\) with respect to (the
usual) stochastic (dominance) order (see for example \cite{Shaked2007}). We need
therefore only study \(\prob(X_{n,\mu} \geq \mu+l)\) for \(l \geq 1\) and
integer \(\mu+l\). These observations are essentially the same as in the
previously mentioned papers of \citeauthor{Greenberg2014} and
\citeauthor{Doerr2018}.

For readability we will restrict ourselves to integer \(\mu\) and \(l\). The
results hold more generally since discontinuities appear only at integer
values of \(\mu+l\). Both Theorem~\ref{thm:binomial1} and the later
Theorem~\ref{thm:poisson1} are essentially special cases of the consequence of
the inequality of Hoeffding\cite{Hoeffding1956} previously used by
\citeauthor{Anderson1967}\cite{Anderson1967}.

\begin{samepage}
  \begin{theorem}
    \label{thm:binomial1}
    For integer \(l=1,2,\dotsc\), \(\mu=1,\dotsc,n-l\) and
    \(X_{n,\hspace{0.3mm}\mu} \sim \pbinom(n, \mu/n)\) it holds that
    \begin{equation*}
      \begin{split}
        1/2
        % \geq
        % \prob(Z_{\mu} \geq \mu+l)
        \geq
        \prob(X_{n,\hspace{0.3mm}\mu} \geq \mu+l)
        &\geq
        \prob(X_{\mu+l,\hspace{0.3mm}\mu} = \mu+l) = \left( \frac{\mu}{\mu+l} \right)^{\mu+l}
        \\
        &\geq
        \prob(X_{1+l,1} = 1+l) = \frac{1}{(1+l)^{1+l}}.
      \end{split}
    \end{equation*}
  \end{theorem}
\end{samepage}

While the upper bound is a well known folklore theorem it also follows directly
from the proof technique. Since the lower bounds on
\(\prob(X_{n,\hspace{0.3mm}\mu} \geq \mu+l)\) explicitly identify the minimising
parameters they are, by construction, sharp.

The problem has a certain amount of symmetry given by the fact that \(n -
X_{n,\hspace{0.3mm}\mu}\) has the same distribution as \(X_{n,n-\mu}\). Any
result concerning \(\prob(X_{n,\mu} \geq \mu + l)\) therefore translates into
one concerning \(\prob(X_{n,\hspace{0.3mm}\mu} \leq \mu - l) = \prob(X_{n,n-\mu}
\geq n - \mu + l) = 1 - \prob(X_{n,n-\mu} \leq n - \mu + l - 1)\).

What follows in Corollary~\ref{cor:binom2} summarises the basic conclusions one
may now derive from Theorem~\ref{thm:binomial1}.
 
\begin{corollary} \label{cor:binom2} For \(l = 0, 1, \dotsc\), \(n = l+1, l+2,
  \dotsc\), and \(X_{n,\mu} \sim \pbinom(n, \mu/n)\) it holds that
  \begin{enumerate}[a)]
  \item for \(1 \leq \mu \leq n-l\) one has \label{cor:binom2part1}
    \begin{equation*}
      % 1/2
      % \geq
      \prob(X_{n, \hspace{0.3mm}\mu} \geq \mu + l)
      \geq
      \left( \frac{\lfloor\mu\rfloor}{\lfloor\mu\rfloor+l+1} \right)^{\lfloor\mu\rfloor+l+1}
      \geq
      \frac{1}{(2+l)^{2+l}},
    \end{equation*}
  \item if \(l \leq \mu \leq n-1\) one has \label{cor:binom2part2}
    \begin{equation*}
      % 1/2
      % \geq
      \prob(X_{n, \hspace{0.3mm}\mu} \leq \mu - l)
      \geq
      \left( \frac{n-\lfloor\mu\rfloor}{n-\lfloor\mu\rfloor+l+1} \right)^{n-\lfloor\mu\rfloor+l+1}
      \geq
      \frac{1}{(2+l)^{2+l}}.
    \end{equation*}
  \end{enumerate}
\end{corollary}

Taking \(l = 0\) in Part~\ref{cor:binom2part1} gives the result of
\citeauthor{Greenberg2014}\cite{Greenberg2014} while \(l = 1\) corresponds to
the result included in \citeauthor{Doerr2018}\cite{Doerr2018}.

Note that the assumption \(\mu \geq 1\) is easily weakened since
\(\prob(X_{n,\mu} \geq \mu+l)\) is monotone increasing on \(\mu \in [0, 1)\).
For example, in the case \(l=0\) one finds that it suffices that \(\mu \geq
\log(3/4)\) for \(\prob(X_{n,\mu} \geq \mu) \geq 1/4\) to hold.

\begin{figure}[h]
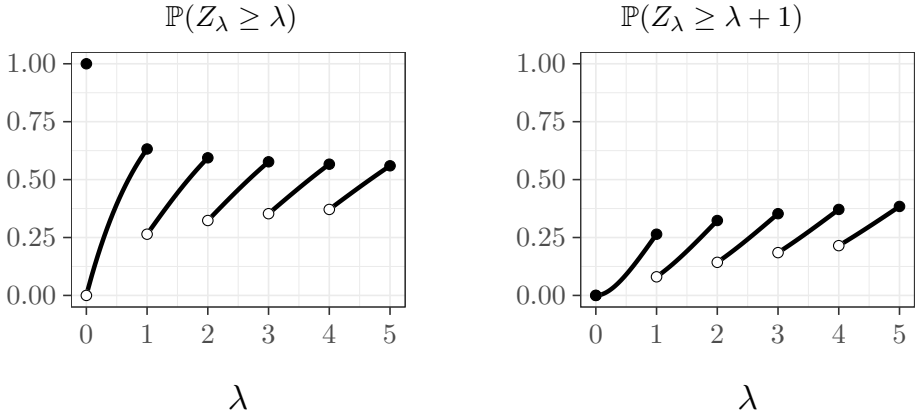

  \centering
  \begin{minipage}{0.46\textwidth}
    \centering{}
    \(\prob(Z_{\lambda} \geq \lambda)\)
    \mbox{\hspace{-1.25cm}\input{pictures/poisplot1.tex}}
  \end{minipage}
  \begin{minipage}{0.46\textwidth}
    \centering{}
    \(\prob(Z_{\lambda} \geq \lambda+1)\)
    \mbox{\hspace{-0.5cm}\input{pictures/poisplot2.tex}}
  \end{minipage}
  \caption{Probability of Poisson random variable exceeding (left) or
    exceeding by at least 1 (right) its expected value \(\lambda\)}
  \label{fig:fig2}
\end{figure}

Essentially the same method as Theorem~\ref{thm:binomial1} and
Corollary~\ref{cor:binom2} yields an analogous bound for the Poisson
distribution.

\begin{theorem} \label{thm:poisson1} For \(Z_{\lambda} \sim \ppois(\lambda)\)
  where \(\lambda \geq 1\) and \(l=0,1,\dotsc\) one has
  \begin{equation*}
    1/2 \geq \prob(Z_{\lambda} \geq \lambda + l) \geq \prob(Z_{1} \geq 1 + l) = 1-\sum_{k=0}^{l}\frac{1}{e k!}.
  \end{equation*}
\end{theorem}

For example, taking \(l = 0\), one finds \(\prob(Z_{\lambda} \geq \lambda) \geq
1 - e^{-1}\).

We will prove these results using the fact that, roughly speaking, were it not
for the discreteness of the distribution it would be the case that
\(\prob(X_{n,\mu} > \mu)\) was increasing in \(n\) when \(p = \mu/n\) was held
fixed. The actual inequalities used are illustrated in
Figure~\ref{fig:binomiallattice}.

\begin{figure}[h]
  \centering
  \begin{tikzpicture}
    \node (highk) {\(\prob(X_{n,\mu} > \mu)\)};
    \node[below left= 1.8cm and 1.8cm of highk.center, anchor=center] (lowkminus) {\(\prob(X_{n\shortminus{}1,\mu\shortminus{}1} > \mu\!-\!1)\)};
    \node[below right= 1.8cm and 1.8cm of highk.center, anchor=center] (lowk) {\(\prob(X_{n\shortminus{}1,\mu} > \mu)\)};
    \node[above right= 1.8cm and 1.8cm of lowk.center, anchor=center] (highkplus) {\(\prob(X_{n,\mu+1} > \mu\!+\!1)\)};
    \node[above left= 1.8cm and 1.8cm of lowkminus.center, anchor=center] (highkminus) {\(\prob(X_{n,\mu\shortminus{}1} > \mu\!-\!1)\)};
    \path (highkminus) -- node[sloped] {\(\geq\)} (lowkminus) -- node[sloped] {\(\leq\)} (highk) -- node[sloped] {\(\geq\)} (lowk) -- node[sloped] {\(\leq\)} (highkplus);
    \node[above= 1cm of highkminus.center, anchor=center] (above1) {\(\frac{\mu-1}{n}\)};
    \node[above= 2.8cm of lowkminus.center, anchor=center] (above2) {\(\frac{\mu-1}{n-1}\)};
    \node[above= 1cm of highk.center, anchor=center] (above3) {\(\frac{\mu}{n}\)};
    \node[above= 2.8cm of lowk.center, anchor=center] (above4) {\(\frac{\mu}{n-1}\)};
    \node[above= 1cm of highkplus.center, anchor=center] (above5) {\(\frac{\mu+1}{n}\)};
    \path (above1) -- node {\(\leq\)} (above2) -- node {\(\leq\)} (above3) -- node {\(\leq\)} (above4) -- node {\(\leq\)} (above5);
    \node[below left= 1.8cm and 1.8cm of highkminus.center,anchor=center] (dummy1) {\(\mathrlap{\ldots}\)};
    \node[below right= 1.8cm and 1.8cm of highkplus.center,anchor=center] (dummy2) {\(\mathllap{\ldots}\)};
    \node[above right=2.8cm and 1cm of dummy1.center, anchor=center] {\(p = \)};
    \path (dummy1) -- node[sloped] {\(\leq\)} (highkminus);
    \path (dummy2) -- node[sloped] {\(\geq\)} (highkplus);
  \end{tikzpicture}
  \caption{Illustration of how the central inequalities (found in
    Corollary~\ref{cor:hoeffding}) for proving Theorem~\ref{thm:binomial1} say,
    roughly, that the probability of exceeding the expectation is monotone in
    \(n\) when \(p = \mu/n\) is fixed.}
  \label{fig:binomiallattice}
\end{figure}

Indeed, with an appropriate interpolation the monotonicity can be made to hold
exactly, at least for certain ranges of \(\mu\) and \(l\).

Order statistics of the uniform distribution on the unit
interval are beta distributed, so that for any integer \(k\) one has
\(\prob(X_{n,\mu} \geq k) = \prob(W_{k,n-k+1} \leq \mu/n)\) where \(W_{a,b} \sim
\pbeta(a, b)\). That is to say \(W_{a,b}\) has a (Lebesgue) absolutely
continuous distribution supported on the unit interval with density \(f_{a,b}(x)
= x^{a-1}(1-x)^{b-1}/B(a, b)\) for \(B(a, b) = \int_{0}^{1}x^{a-1}(1-x)^{b-1} =
\Gamma(a)\Gamma(b)/\Gamma(a+b)\) the Beta function. Combining this fact with the
fact that \(\pbeta(a, b)\) is stochastically increasing in \(a\) and decreasing
in \(b\) yields the following.

\begin{proposition} \label{prop:binombetabound}
  For \(X_{n,\mu} \sim \pbinom(n, \mu/n)\), \(W_{a,b} \sim \pbeta(a, b)\) and
  any \(n \geq 1\), \(\mu \in [0, n]\) and \(0 < \mu + l < n\), the following inequalities hold
  \begin{equation*}
    \prob(W_{\mu+l,n-\mu-l+1} \geq \mu/n)
    \leq
    \prob(X_{n,\mu} \geq \mu + l)
    <
    \prob(W_{\mu+l+1,n-\mu-l} \geq \mu/n),
  \end{equation*}
  with the lower bound being an equality whenever \(\mu\) is an integer, and the upper
  bound being tight as \(\mu\) approaches an integer from below.
\end{proposition}

One may thus think of \(\prob(W_{\mu+l,n-\mu-l+1} \geq \mu/n)\) and
\(\prob(W_{\mu+l+1,n-\mu-l} \geq \mu/n)\) as lower and upper smooth
interpolations of \(\prob(X_{n,\mu} \geq \mu + l)\). Having gotten rid of the
discontinuities we study instead these smoothly varying quantities. They tend, at
least under certain conditions, \emph{monotonically} to \(1/2\) as \(n\) tends
to infinity.

\begin{theorem} \label{thm:betamono} Let \(p \in (0, 1)\), \(a, b \in \R\). For
  \(n \in \R\) such that \(-a/p, -b/(1-p) < n\) let \(W_{n} \sim \pbeta(pn+a,
  (1-p)n+b)\).

   If one of the following two conditions hold:
   \begin{itemize}
   \item \(p \leq 1/2\) and either \(a \geq 1 \geq b\) or
     \((a+b-2)((a-1)/(a+b-2) - p) > 0\),
   \item \(p \geq 1/2\) and either \(b \geq 1 \geq a\) or
     \((a+b-2)((a-1)/(a+b-2) - p) < 0\),
   \end{itemize}
   then \(\prob(W_{n} \leq p)\) tends monotonically to \(1/2\) as \(n\) tends to
   infinity.
\end{theorem}

We conjecture that the statement is true even without the conditions on \(p\),
\(a\) and \(b\) holding. In case \(a = b = 1\) the theorem applies to all \(p\)
and \(p\) is equal to the mode of the distribution. Taking, for example, the
case \(n = 0\) and the case of \(n\) tending to \(\infty\) shows the probability
of of being less than or greater than some fixed mode \(p\) lies in the interval containing
\(p\) and \(1/2\) when both both parameters are greater than \(1\).

Combining Proposition~\ref{prop:binombetabound}~and~Theorem~\ref{thm:betamono}
one can produce variants of Proposition~\ref{prop:binombetabound} uniform in
\(n\), such as the following.
\begin{corollary} \label{cor:ineqlat}
  For any \(n \geq 1\), \(\mu \in [0, n]\), and \(a, b > 0\) let \(X_{n,\mu}
  \sim \pbinom(n, \mu/n)\) and \(W_{a,b} \sim \pbeta(a, b)\). Then if \(\mu/n
  \leq 1/2\), \(l = 1, 2, \dotsc\), and \(\mu + l < n\) it holds that
  \begin{equation*}
    \prob(X_{n,\mu} \geq \mu + l)
    \geq \prob(W_{\mu/n+l,(1-\mu/n)-l+1} \geq \mu/n).
  \end{equation*}
\end{corollary}

\section{Proofs}

The workhorse in the proofs of
Theorem~\ref{thm:binomial1}~and~\ref{thm:poisson1} is the following old result
due to Hoeffding.

\begin{theorem}[Theorem~5 of
  \citeauthor{Hoeffding1956}\cite{Hoeffding1956}] \label{thm:hoeffding} Fix some
  \(n = 1,2,\dotsc\) let \(Y_{1}, \dotsc, Y_{n}\) be independent indicators with
  \(\prob(Y_{i} = 1) = p_{i}\) and \(\sum_{i=1}^{n}p_{i} = \mu\). For any \(0
  \leq a \leq \mu-1 < \mu \leq b \leq n\) it holds that
  \begin{equation} \label{eq:hoeffding}
    \prob(a \leq \pbinom(n,\mu/n) \leq b)
    \leq
    \prob(a \leq \sum_{i=1}^{n}Y_{i} \leq b).
  \end{equation}
\end{theorem}
It was later shown by \citeauthor{Gleser1975}\cite{Gleser1975} that this is a
special case of a more general majorization inequality, the right hand side of
Equation~\ref{eq:hoeffding} is really Shur-convex as a function of the vector
\((p_{1}, \dotsc, p_{n})\).

The following corollary
appears also in \citeauthor{Anderson1967}\cite{Anderson1967}, to compare
binomial and Poisson probabilities, and in
\citeauthor{Jogdeo1968}\cite{Jogdeo1968}, where it is used to analyse the median
of the binomial distribution. Since it is short and instructive, we include a
proof for completeness.

\begin{corollary} \label{cor:hoeffding}

  For integer \(a > \mu\) and \(X_{n, \mu} \sim \pbinom(n, \mu/n)\) one has
  \begin{align}
    \prob(X_{n,\mu} \geq a) &\geq \prob(X_{n-1,\mu-1} \geq a-1) \label{eq:hoeffding1}
    \intertext{and}
    \prob(X_{n,\mu} \geq a) &\geq \prob(X_{n-1,\mu} \geq a),\vspace{-0.1cm} \label{eq:hoeffding2}
  \end{align}
  with inequalities reversed if \(\mu \geq a\).
\end{corollary}
\begin{proof}
  Let \(Y_{1}, \dotsc, Y_{n}\) be as in the statement of
  Theorem~\ref{thm:hoeffding}. Note that \(X_{n-1,\mu}\) and \(X_{n-1,\mu-1}+1\)
  have the same distribution as \(\sum_{i=1}^{n}Y_{i}\) if the vector \((p_{1},
  \dotsc, p_{n})\) is chosen as \((0, \mu/(n-1), \dotsc, \mu/(n-1))\) or \((1,
  (\mu-1)/(n-1), \dotsc, (\mu-1)/(n-1))\), respectively. The statement now
  follows from Theorem~\ref{thm:hoeffding} by taking \(a = 0\) and \(b = l-1\)
  in Equation~\ref{eq:hoeffding}.
\end{proof}

Using vectors of the type \((0, p, \dotsc, p)\) and \((1, p, \dotsc, p)\) to
extract useful inequalities this way from the Shur-convexity of functions is not
novel. Using a different majorization inequality, it was applied by
Proschan\cite{Proschan1965} to show a somewhat similar argument saying that, in
some circumstances, the convergence of averages of random variables to the
limiting mean is monotone. Other examples can be found in, for example, the book
of Karlin[p.289~and~387]\cite{Karlin1968}.

\begin{proof}[Proof of Theorem~\ref{thm:binomial1}]
  The proof is now a matter of following the ``lattice'' of inequalities
  in Figure~\ref{fig:binomiallattice} (established in Corollary~\ref{cor:ineqlat}) downward.

  Let \(l \geq 1\) and \(\mu = 1, \dotsc, n-l\) as in the statement. Iteratively 
  applying Equation~\ref{eq:hoeffding1} gives
  \begin{equation*}
    \prob(X_{n,\mu} \geq \mu+l) \geq \prob(X_{n-1,\mu} \geq \mu+l) \geq \cdots \geq \prob(X_{\mu+l,\mu} \geq \mu+l).
  \end{equation*}
  Since \(\prob(X_{\mu+l,\mu} \geq \mu+l) = \prob(X_{\mu+l,\mu} = \mu+l) =
  (\mu/(\mu+l))^{\mu+l}\) this gives the first lower bound of the statement.
  Similarly applying Equation~\ref{eq:hoeffding2} now gives the second one,
  since
  \begin{equation*}
    \prob(X_{\mu+l,\mu} \geq \mu+l) \geq \prob(X_{\mu+l-1,\mu-1} \geq \mu+l-1) \geq \cdots \geq \prob(X_{l+1,1} \geq 1+l).
  \end{equation*}
  Which is the desired bound since
  \begin{equation*}
  \prob(X_{l+1,1} \geq 1+l) = \prob(X_{l+1,1}
  = l+1) = (1/(l+1))^{l+1}.
\end{equation*}
  The upper bound follows by iterating Equation~\ref{eq:hoeffding1} in the
  reverse direction, so that
  \begin{equation*}
    \prob(X_{n,\mu} \geq \mu+l)
    \leq
    \lim_{k \to \infty}\prob(X_{n+k,\mu+k} \geq \mu+k+l)
    =
    \prob(W \geq \mu) = 1/2.
  \end{equation*}
  where \(W \sim \pnorm(\mu, 1)\).\qedhere{}
\end{proof}

To get the general case of the corollary one need now only use domination of
\(X_{n,\mu}\) for \(\mu\) increasing and be slightly careful to round \(\mu\) in
the correct direction.

\begin{proof}[Proof of Corollary~\ref{cor:binom2}]
  Let \(l \geq 0\) be an integer, as in the statement. Then
  \begin{equation*}
    \begin{split}
      \prob(X_{n, \mu} \geq \mu + l)
      &=
      \prob(X_{n, \mu} \geq \lceil \mu + l \rceil)
      \\&\geq
      \prob(X_{n, \mu} \geq \lfloor \mu \rfloor + l + 1)
      \geq
      \prob(X_{n, \lfloor\mu\rfloor} \geq \lfloor \mu \rfloor + l + 1)
    \end{split}
  \end{equation*}
  where the first inequality is an equality when \(\mu\) is an integer and the
  second inequality follows from \(X_{n,\mu}\) stochastically dominating
  \(X_{n,\mu'}\) whenever \(\mu \geq \mu'\) (see for example \cite{Doerr2018}).
  The first part of the statement is now a direct application of
  Theorem~\ref{thm:binomial1} and the second part follows analogously.
\end{proof}

The main part of the following proof is essentially one for Corollary~2.2 of
\citeauthor{Anderson1967}\cite{Anderson1967}. Since they do not include an
explicit proof we give one here.

\begin{proof}[Proof of Theorem~\ref{thm:poisson1}]
  Let \(\lambda \geq 1\) and \(l \geq 1\) be integers, take \(\mu = \lambda <
  \lambda + l = a\) in Equation~\ref{eq:hoeffding1}, and let \(n\) tend to
  infinity. This gives
  \begin{equation*}
    \lim_{n \to \infty}\prob(X_{n,\lambda} \geq \lambda+l)
    \geq
    \lim_{n \to \infty}\prob(X_{n-1,\lambda-1} \geq \lambda-1+l).
  \end{equation*}
  Both \(X_{n,\lambda}\) and \(X_{n-1,\lambda-1}\) tend to Poisson distributed
  limits with expected values \(\lambda\) and \(\lambda-1\),
  respectively. In other words we have
  \begin{equation*}
    \prob(Z_{\lambda} \geq \lambda+l)
    \geq
    \prob(Z_{\lambda-1} \geq \lambda-1+l).
  \end{equation*}
  Iterating this inequality we get the desired lower bound \(\prob(Z_{\lambda}
  \geq \lambda+l) \geq \prob(Z_{1} \geq 1+l) = 1 - \sum_{k=1}^{l}1/(k!e)\).

  Similarly for \(\lambda > l \geq 1\) plugging \(\mu = \lambda \geq \lambda - l
  + 1 = a\) into Equation~\ref{eq:hoeffding1} and letting \(n\) tend to infinity
  gives
  \begin{equation*}
    \begin{split}
      \prob(Z_{\lambda} \leq \lambda-l)
      &=
      1-\prob(Z_{\lambda} \geq \lambda-l+1)
      \\&\geq
      1-\prob(Z_{\lambda-1} \geq \lambda-l)
      \\&=
      \prob(Z_{\lambda-1} \geq \lambda-1-l).
    \end{split}
  \end{equation*}
  Iterating, again, gives the second bound \(\prob(Z_{\lambda} \leq \lambda-l)
  \geq \prob(Z_{l} = 0) = e^{-l}\).
\end{proof}

We now turn to the proof of Theorem~\ref{thm:betamono}, starting with
establishing the concentration behaviour of \(\pbeta(a, b)\) as \(a\) and \(b\)
both increase.

\begin{lemma} \label{lemma:prepowerdist} Let \(X\) and \(Y\) be random variables
  on some measurable space \(\mathcal{X}\) with densities \(f_{X}\) and
  \(f_{Y}\) relative to a reference measure \(\mu\).

  If \(f \colon \mathcal{X} \to \R\) and \(C \in \R\) are such that \(f_{X}(x) >
  f_{Y}(x)\) if and only if \(f(x) > C\) then for all \(t \in \R\)
  \begin{equation*}
      \prob(f(X) > t) \geq \prob(f(Y) > t).
  \end{equation*}
\end{lemma}

\begin{proof}
  For any \(t > C\) we have
  \begin{equation*}
    \begin{split}
      \prob(f(X) > t) &= \int \mathbf{1}_{f > t}(x) f_{X}(x) \,\mu(dx) \\&\geq \int \mathbf{1}_{f > t}(x) f_{Y}(x) \,\mu(dx) = \prob(f(Y) > t),
    \end{split}
  \end{equation*}
  where \(\mathbf{1}_{f > t}(x) = 1\) if \(f(x) > t\) and \(0\) otherwise.

Similarly for \(t \leq C\)
\begin{equation*}
    \prob(f(X) > t) = 1 - \prob(f(X) \leq t)
    \geq 1 - \prob(f(Y) \leq t) = \prob(f(Y) > t). \qedhere
\end{equation*}
\end{proof}

Our interest lies in the following special case.

\begin{lemma} \label{lemma:powerdist1} Let \(X\) and \(Y\) be random variables
  with values in some measurable space \(\mathcal{X}\) with densities \(f_{X}(x)
  = C_{X}g_{X}(f(x))\) and \(f_{Y}(x) = C_{Y}g_{Y}(f(x))\) relative to some
  reference measure \(\mu\), some positive measurable \(g_{X},g_{Y} \colon \R
  \to \R\), \(f \colon \mathcal{X} \to \R\), and normalising constants \(C_{X},
  C_{Y} \in \R\).

  If \(g_{X}/g_{Y}\) is increasing, then for any \(t \in \R\) it holds that
  \begin{equation*}
      \prob(f(X) > t) \geq \prob(f(Y) > t).
  \end{equation*}
\end{lemma}
\begin{proof}
  Denote by \(G^{-}\) the, necessarily, increasing (generalised) inverse of
  \(g_{X}/g_{Y}\). Then \[f_{X}(x) = C_{X}g_{X}(f(x)) \leq f_{Y}(x) =
    C_{Y}g_{Y}(f(x))\] if and only if \[g_{X}(f(x))/g_{Y}(f(x)) \leq
    C_{Y}/C_{X}\] if and only if \(f(x) \leq C = G^{-}(C_{Y}/C_{X})\). The
  statement now follows from Lemma~\ref{lemma:prepowerdist}.
\end{proof}

The special case of Lemma~\ref{lemma:powerdist1} with \(g_{X}(x) = x\),
\(g_{Y}(x) = 1\) and \(\mu = g \nu\) for some positive, measurable \(g\) and
other reference measure \(\nu\) corresponds to the case with densities
\(f_{X}(x) = C_{X}f(x)g(x)\) and \(f_{Y}(x) = C_{Y}g(y)\) (relative to \(\nu\)).

By passing to conditional distributions, Lemma~\ref{lemma:powerdist1}
generalises to conditional probabilities.

\begin{lemma} \label{lemma:powerdist2} Let \(X\) and \(Y\) be random variables
  with values in some measurable space \(\mathcal{X}\) with distributions given
  by densities \(f_{X}(x) = C_{X}g_{X}(f(x))\) and \(f_{Y}(x) =
  C_{Y}g_{Y}(f(x))\) relative to some reference measure \(\mu\), some positive
  measurable \(g_{X},g_{Y} \colon \R \to \R\), \(f \colon \mathcal{X} \to \R\),
  and normalising constants \(C_{X}, C_{Y} \in \R\).

  If \(g_{X}/g_{Y}\) is increasing, then for any \(t \in \R\) and measurable \(A
  \subset \R\) such that \(\prob(X \in A), \prob(Y \in A) > 0\) it holds that
  \begin{equation*}
      \prob(f(X) > t \mid X \in A) \geq \prob(f(Y) > t \mid Y \in A).
  \end{equation*}
\end{lemma}
\begin{proof}
  There exist random variables \(X'\) and \(Y'\), possibly on some other
  underlying space, with values in \(\mathcal{X}\) and with distributions given
  by densities \(f_{X'}(x) = C_{X'}g_{X}(f(x))\) and \(f_{Y'}(x) =
  C_{Y'}g_{Y}(f(x))\) relative to \(\mu'\), the restriction of \(\mu\) to \(A\).
  Necessarily \(C_{X'} = C_{X}/\prob(X \in A)\) and \(C_{Y'} = C_{Y}/\prob(Y \in
  A)\).

  By construction \(X'\) and \(Y'\) are such that for any measurable \(B\) one
  has \(\prob(X \in B \mid X \in A) = \prob(X' \in B)\) and \(\prob(Y \in B \mid
  Y \in A) = \prob(Y' \in B)\). In particular \(\prob(f(X) > t \mid X \in A) =
  \prob(f(X') > t)\) and \(\prob(f(Y) > t \mid Y \in A) = \prob(f(Y') > t)\).
  The statement follows by applying Lemma~\ref{lemma:powerdist1} to \(X'\) and
  \(Y'\).
\end{proof}

This conditional version becomes useful, to us, by considering sets \(A\) where
\(f\) is monotone.

\begin{lemma}\label{lemma:powerdist3}
  For \(c \geq a > 0\) and \(d \geq b > 0\) let \(X \sim \pbeta(a, b)\) and \(Y
  \sim \pbeta(c, d)\). Introducing \(p = (c-a)/((c+d)-(a+b))\) it holds for
  any \(t \leq s \leq p\)that
    \begin{equation*}
        \prob(X < t \mid X < s) \leq \prob(Y < t \mid Y < s)
    \end{equation*}
    while for any \(p \leq s \leq t\) we have
    \begin{equation*}
        \prob(X > t \mid X > s) \leq \prob(Y > t \mid Y > s).
    \end{equation*}
\end{lemma}
\begin{proof}
  Introduce \(f(x) = x^p(1-x)^{1-p}\) and let \(\mu\) be the (finite) measure
  given by the (Lebesgue) density \(h(x) = x^{a-1}(1-x)^{b-1}f(x)^{m}\) where \(m =
  a+b\). Relative to \(\mu\) the distributions of \(X\) and \(Y\) have densities
  \(f_{X}(x) = c_{X}f(x)^{n-m}\) and \(f_{Y}(x) = c_{Y}\) for appropriate
  normalising constants \(c_X\) and \(c_Y\) and where \(n = c+d\). Since
  \(x^{n-m}\) is increasing in \(x\) when \(n>m\) it follows from
  Lemma~\ref{lemma:powerdist2} that
  \begin{equation} \label{eq:powerdist3main}
      \prob(f(X) > t \mid X \in A) \geq \prob(f(Y) > t \mid Y \in A)
  \end{equation}
  for any measurable \(A \subset \R\) and \(t \in \R\).

  Note that \(f\) is (strictly) increasing on \([0, p]\) and denote by
  \(f^{-1}_{-}\) its inverse on \([0, f(p)]\). Applying
  Equation~\ref{eq:powerdist3main} to \(A = [0, s]\) for any \(s \leq p\) yields
  \begin{equation*}
    \prob(X > f^{-1}_{-}(t) \mid X < s) \geq \prob(Y > f^{-1}_{-}(t) \mid Y < s).
  \end{equation*}
  Since \(t \in \R\) is arbitrary we may conclude that
  \begin{equation*}
    \prob(X < t \mid X < s) \leq \prob(Y < t \mid Y < s).
  \end{equation*}

  Similarly, \(f\) is (strictly) decreasing on \([p, 1]\) so that following the
  same steps with \(A = [s, 1]\) for any \(s \geq t\) yields
  \begin{equation*}
    \prob(X > t \mid X > s) \leq \prob(Y > t \mid Y > s). \qedhere
  \end{equation*}
\end{proof}

Lemma~\ref{lemma:powerdist3} shows that increasing \(a\) to \(c\) and \(b\) to
\(d\) shifts mass away from both tails of the beta distribution towards \(p\).
The remainder of the proof of Theorem~\ref{thm:betamono} will be concerned with
translating this to a statement concerning the relative size of \(\prob(X < p)\)
to \(\prob(X > p)\). Here the main idea will be that for certain parameters the
major contribution to non-equality of these two probabilities occurs closer to
the tails (at \(0\) and \(1\)), away from \(p\). In these cases, concentrating
the distribution at \(p\) will therefore make the two probabilities more equal,
meaning closer to \(1/2\).

\begin{lemma}\label{lemma:oddsexpect}
  Let \(r \colon [0, 1] \to [0, 1]\) be a smooth, strictly decreasing involution
  with a fixed point at \(p \in [0, 1]\). In other words, \(r(r(x)) = x\) for
  all \(x \in [0, 1]\) and \(r(p) = p\).

  If \(X\) is a random variable on \([0, 1]\) with a distribution given by a
  (Lebesgue) density \(f_{X}(x) = f(x)g(x)\) where \(g > 0\) and \(f(r(x)) =
  f(x)\) for all \(x \in [0, 1]\). Then
  \begin{align*}
    \frac{\prob(X \geq p)}{1-\prob(X \geq p)}
    &=
      \expect\Big(-r'(X)\frac{g(r(X))}{g(X)} \, \Big| \, X < p \Big)
      \intertext{and}
      \frac{\prob(X \leq p)}{1-\prob(X \leq p)}
    &=
      \expect\Big(-r'(X)\frac{g(r(X))}{g(X)} \, \Big| \, X > p \Big).
  \end{align*}
\end{lemma}
\begin{proof}
  The statement follows by a simple calculation involve a change of variables
  using the decreasing involution \(r\):
    \begin{equation*}
        \begin{split}
            \prob(X \geq p)
            &=
            \int_{p}^{1} f_{X}(x) \, dx
            \\&=
            \int_{p}^{1} g(x)f(x) \, dx
            \\&=
            \int_{0}^{p} g(r(x))f(r(x))(-r'(x)) \, dx
            \\&=
            \int_{0}^{p} -r'(x)\frac{g(r(x))}{g(x)} g(x)f(x) \, dx
            \\&=
            \int_{0}^{p} -r'(x)\frac{g(r(x))}{g(x)} f_{X}(x) \, dx.
        \end{split}
    \end{equation*}
    Dividing both sides by \(1-\prob(X \geq p) = \prob(X < p)\) gives the first
    equation. The second one follows similarly.
\end{proof}

Recall that whenever \(X\) stochastically dominates \(Y\), then \(\expect(h(X))
\geq \expect(h(Y))\) for any increasing \(h\). The representation in
Lemma~\ref{lemma:oddsexpect} becomes useful for comparing distributions of
random variables \(X\) and \(Y\) with a common \(r\), where \(g\) differs only
by multiplication with a constant, one variable conditionally dominates the
other, and the function inside the expression is monotone on at least one of the
two intervals \([0, p]\) and \([p, 1]\).

Visually, the existence of a pair of functions \(f\) and \(r\) as in the
statement of Lemma~\ref{lemma:oddsexpect} means \(f\) is (anti-)unimodal with
\(r\) reflecting a point \(x \in [0, 1]\) around the (anti-)mode \(p\) of \(f\).
We begin by working out the behaviour of the reflection function \(r\) in the
case of the family of Beta distributions.

\begin{lemma} \label{lemma:rprops} For \(a, b \in \R \setminus \{0\}\) of the
  same sign there exists a unique continuous decreasing function \(r_{p} \colon
  [0, 1] \to [0, 1]\), depending only on \(p = a/(a+b) \in (0, 1)\), such that
\begin{equation} \label{eq:reflectme}
    x^{a}(1-x)^{b} = r_{p}(x)^{a}(1-r_{p}(x))^{b}
\end{equation}
for \(x \in (0, 1)\).

Moreover, the family of functions \(r_{p}\) for \(p \in (0, 1)\) thus defined
satisfy the following properties:
\begin{enumerate}[label=\arabic*),ref=\arabic*]
  \item \(r_{p}\) is a smooth involution with a fixed point at \(p\), meaning
  \(r_{p}(r_{p}(x)) = x\) and \(r_{p}(p) = p\). In particular \(r_{p}(0) = 1\)
  and \(r_{p}(1) = 0\). \label{prop:rbasics}
 % \item For all \(x \in [0, 1]\) one has \(r_{1-p}(x) = 1 -  r_{p}(1-x)\). \label{prop:rpropsreflect}
  \item \label{prop:rincp} If \(p < q\) then \(r_{p}(x) < r_{q}(x)\) for all \(x
    \in (0, 1)\).
  \item \label{prop:rpropsbasicskew} If \(p < 1/2\) then \(2p - x \leq r_{p}(x)
    < 1 - x\) for \(x \in (0, 1)\). For \(p > 1/2\) the inequalities are
    reversed and for \(p = 1/2\) the inequalities are equalities.
  \item \(r_{p}\) satisfies the differential equation  \label{prop:rpropsreflectderiv}
  \[(p-r_{p}(x))x(1-x)r_{p}'(x)
  =
  (p-x)r_{p}(x)(1-r_{p}(x))\]
  for \(x \in (0, 1)\).
  \item \(r_{p}\) satisfies the differential equation 
    \label{prop:rpropsreflectderiv2}
  \begin{equation*}
  \begin{split}
      &(p-x)x(1-x)(p-r_{p}(x))^{2}r_{p}''(x)
  %=
  %f_{p}'(x)/f_{p}'(r_{p}(x))
  \\&\quad=
  p(1-p)(2p - x - r_{p}(x))(r_{p}(x) - x)r_{p}'(x)
  \end{split}
  \end{equation*}
  for \(x \in (0, 1)\).
  \end{enumerate}
\end{lemma}
\begin{proof}
  Introduce \(f_{p}(x) = x^{p}(1-x)^{1-p}\). Note that \(x^{a}(1-x)^{b} =
  (x^{p}(1-x)^{1-p})^{a+b}\) when \(p = a+b\). In particular \(x^{a}(1-x)^{b} =
  y^{a}(1-y)^{b}\) if and only if \(f_{p}(x) = f_{p}(y)\). For \(r_{p}\) as in
  the statement it suffices therefore to consider Equation~\ref{eq:reflectme}
  for \(a = p\) and \(b = 1-p\), meaning \(f_{p}(r_{p}(x)) - f_{p}(x) = 0\).
  Existence now follows from the fact that \(f_{p}\) is strictly increasing on
  \([0, p]\) and then strictly decreasing on \([p, 1]\). Define \(r_{p}(x) =
  f^{-1}(f(x))\) where the inverse is taken on either \((p, 1)\) or \((0, p]\)
  depending on whether \(x \in (0, p)\) or \(x \in [p, 1)\). Clearly \(r_{p}\)
  satisfies Equation~\ref{eq:reflectme}. Smoothness and uniqueness follows by
  the implicit function theorem outside of \(x = p\) and at \(x = p\) by
  continuous extension.

  Similarly, the identity function is the unique \emph{increasing} function
  satisfying Equation~\ref{eq:reflectme}. This means in particular that
  \(r_{p}(r_{p}(x))\) must be the identify function.

  This proves the basic statement as well as Property~\ref{prop:rbasics}.

  For Property~\ref{prop:rincp} note that \(r_{p}(x) - x\) has the same sign as
  \(p - x\) and that, similarly, \(r_{q}(x) - x\) has the same sign as \(q -
  x\). We proceed now by three cases, depending on whether \(p \leq x \leq q\),
  \(x < p\), or \(x > q\).

  For \(x \in [p, q]\) we have \(r_{p}(x) < x < r_{q}(x)\) since \(x \leq
  r_{p}(x)\) if and only if \(x \leq p\) and \(x \leq r_{q}(x)\) if and only if
  \(x \leq q\).

  If \(x < p\) then \(r_{p}(x) > p\) and \(r_{q}(x) > q > p\). It follows that
  \(r_{p}(x) < r_{q}(x)\) if and only if \(f_{p}(x) > f_{p}(r_{q}(x)) =
  (r_{q}(x)/(1-r_{q}(x))^{p-q}f_{q}(x)\) if and only if \((x/(1-x))^{p-q} >
  (r_{q}(x)/(1-r_{q}(x))^{p-q}\) if and only if \(x/(1-x) <
  r_{q}(x)/(1-r_{q}(x)\) if and only if \(x < r_{q}(x)\) if and only if \(x <
  q\), which holds since \(x < p < q\).

  The case \(x > q\) follows analogously to the case \(x <
  p\).

  For Property~\ref{prop:rpropsbasicskew} note that \(r_{1/2}(x) = 1-x\) since
  \(f_{1/2}\) is symmetric. The upper bound now follows from
  Property~\ref{prop:rincp} since \(r_{p}(x) < r_{1/2}(x) = 1-x\) when \(p <
  1/2\). The inequality reversed when \(p > 1/2\).

  The lower bound for \(p < 1/2\) is trivial when \(x > 2p\). For \(x < 2p\) it
  may be rewritten as \(p - y \leq r_{p}(p + y)\) for \(y \in [-p, p]\). If \(y
  > 0\) then \(p - y < p\) and \(r_{p}(p + y) < p\). The statement is therefore
  equivalent to \(f_{p}(p - y) \leq f_{p}(p + y)\). If \(y < 0\) it is
  equivalent to \(f_{p}(p - (-y)) \leq f_{p}(p + (-y))\), so that it suffices to
  consider the case \(y > 0\).

  Expanding the definition \(f_{p}\) the inequality may be rewritten
  as \[\Big(\frac{p-y}{p+y}\Big)^{p} \leq \Big(\frac{1-p-y}{1-p+y}\Big)^{1-p}.\]
  Since \(p < 1-p\) it suffices that the left hand side be non-decreasing as a
  function of \(p \in [y, 1/2]\). Taking the logarithmic derivative and
  introducing \(s = (p-y)/(p+y)\) we see that it suffices to prove
  that \[\frac{1/s - s}{2} \geq -\log(s)\] for \(s \in (0, 1)\). This elementary
  inequality can be verified by observing that it is in equality when \(s = 1\)
  and then comparing derivatives on \(s \in (0, 1)\).

  For Property~\ref{prop:rpropsreflectderiv} the equation is trivial if \(x =
  p\), so assume \(x \neq p\). We begin by differentiating the equation
  \(f_{p}(r_{p}(x)) = f_{p}(x)\), yielding \(f_{p}'(r_{p}(x))r_{p}'(x) =
  f_{p}'(x)\). Next note that \(f_{p}\) has logarithmic derivative \(h_{p}(x) =
  \frac{p-x}{x(1-x)}\), so that \(f_{p}'(x) = f_{p}(x)h_{p}(x)\), for \(x \in
  (0, 1)\). In particular \(f_{p}'(r_{p}(x)) = f_{p}(x)h_{p}(r_{p}(x))\).
  Plugging this into the equation above and cancelling \(f_{p}(x)\) yields a
  first order differential equation
  \begin{equation}\label{eq:rfodeq}
      h_{p}(r_{p}(x)) r_{p}'(x) = h_{p}(x)
  \end{equation}
  or in other words
  \begin{equation*}
  \frac{p-r_{p}(x)}{r_{p}(x)(1-r_{p}(x))} r_{p}'(x) = \frac{p-x}{x(1-x)}.
  \end{equation*}
  Multiplying by \(x(1-x)r_{p}(x)(1-r_{p}(x))\) gives the desired equation.

  Continuing with Property~\ref{prop:rpropsreflectderiv2} we differentiate
  Equation~\ref{eq:rfodeq} which gives \[h_{p}(r_{p}(x))r_{p}''(x) = h_{p}'(x) -
    h_{p}'(r_{p}(x))(r_{p}'(x))^{2}.\] By direct computation one finds that the
  function \(h_{p}\), itself, satisfies \(h_{p}'(x) = -\Big(1 +
  \frac{p(1-p)}{(p-x)^{2}}\Big)h_{p}(x)^{2}\). Substituting this into the above and
  using Equation~\ref{eq:rfodeq} we find
  \begin{equation*}
  \begin{split}
    h_{p}(r_{p}(x))r_{p}''(x)
    &=
    h_{p}'(x)
    -
    h_{p}'(r_{p}(x))(r_{p}'(x))^{2}
    \\&=
    -\Big(1 + \frac{p(1-p)}{(p-x)^{2}}\Big)h_{p}(x)^{2}
    +\Big(1 + \frac{p(1-p)}{(p-r_{p}(x))^{2}}\Big)h_{p}(x)^{2}
    \\&=
    p(1-p)\frac{(p-x)^{2} - (p-r_{p}(x))^{2}}{(p-x)^{2}(p-r_{p}(x))^{2}}h_{p}(x)^{2}
    \\&=
    p(1-p)\frac{(2p-x-r_{p}(x))(r_{p}(x)-x)}{(p-x)^{2}(p-r_{p}(x))^{2}}h_{p}(x)^{2}.
  \end{split}
  \end{equation*}
  Multiplying both sides by \(r_{p}'(x)/h_{p}(x)\) and using
  Equation~\ref{eq:rfodeq} gives
  \begin{equation*}
  \begin{split}
    r_{p}''(x)
    &=
    p(1-p)\frac{(2p-x-r_{p}(x))(r_{p}(x)-x)}{(p-x)^{2}(p-r_{p}(x))^{2}}h_{p}(x)r_{p}'(x)
    \\&=
    p(1-p)\frac{(2p-x-r_{p}(x))(r_{p}(x)-x)}{(p-x)x(1-x)(p-r_{p}(x))^{2}}r_{p}'(x).
  \end{split}
  \end{equation*}
  A multiplication by \((p-x)x(1-x)(p-r_{p}(x))^{2}\) concludes the proof.
\end{proof}

The above lemma allows us to determine when \(r_{p}\) is convex or concave,
which in turn allows us to determine monotonicity of \(r_{p}'\).

\begin{lemma} \label{lemma:rmono} For \(p \in (0, 1)\) let \(r_{p} \colon [0, 1]
  \to [0, 1]\) be the function defined in Lemma~\ref{lemma:rprops}. If \(p \leq
  1/2\) then \(r_{p}\) is convex, if \(p \geq 1/2\) then \(r_{p}\) is concave.
\end{lemma}
\begin{proof}
  By Property~\ref{prop:rpropsreflectderiv2} we have
  \begin{equation*}
  \begin{split}
      &(p-x)x(1-x)(p-r_{p}(x))^{2}r_{p}''(x)
  \\&\quad=
  p(1-p)(2p - x - r_{p}(x))(r_{p}(x) - x)r_{p}'(x).
  \end{split}
  \end{equation*}
  Note that \(p-x\) has the same sign as \(r_{p}(x) - x\),
  \(x(1-x)(p-r_p(x))^{2} > 0\) outside of \(x = p\), and \(r_{p}'(x) < 0\). When
  \(x \neq p\) the sign of \(r_{p}''(x)\) is therefore the same as that of \(x +
  r_{p}(x) - 2p\).

  By Property~\ref{prop:rpropsbasicskew} of Lemma~\ref{lemma:rprops} we have for \(p < 1/2\) that
  \begin{equation*}
      x + r_{p}(x) - 2p \geq x + 2p - x - 2p = 0.
  \end{equation*}
  The inequality is reversed when \(p > 1/2\) and an equality when \(p = 1/2\).
  The case \(x = p\) follows by continuity.
\end{proof}

Having thus dealt with the factor \(-r'(x)\) in the statement of
Lemma~\ref{lemma:oddsexpect} we turn to the factor \(g(r(x))/g(x)\).

\begin{lemma} \label{lemma:hdeq} For \(p \in (0, 1)\) let \(r_{p} \colon [0, 1]
  \to [0, 1]\) be the function defined in Lemma~\ref{lemma:rprops} and let
  \(g(x) = cx^{a}(1-x)^{b}\) for \(x \in (0, 1)\) and \(a, b, c \in \R\) such
  that \(c \neq 0\) and \(a \neq -b\). Then \(h(x) = \frac{g(r_{p}(x))}{g(x)}\)
  satisfies the differential equation
    \begin{equation*}
        (p-r_{p}(x))x(1-x)h'(x)
        =
        (a+b)(q - p)(r_{p}(x) - x)h(x)
    \end{equation*}
    where \(q = a/(a+b)\).
\end{lemma}
\begin{proof}
  First note that the equation is trivial for \(x = p\) and that \(g\) itself
  has a logarithmic derivative \(k(x) = (a+b) \frac{q - x}{x(1-x)}\) where \(q =
  a/(a+b)\) and \(a + b \neq 0\) by assumption. In other words, \(g'(x) =
  h(x)g(x)\). Now assume \(x \neq p\)
    \begin{equation*}
        h'(x)
        =
        \frac{g'(r_{p}(x))r_{p}'(x)g(x) - g(r_{p}(x))g'(x)}{g(x)^2}
        =
        (k(r_{p}(x))r_{p}'(x) - k(x))h(x).
    \end{equation*}
    using Property~\ref{prop:rpropsreflectderiv} from Lemma~\ref{lemma:rprops}
    we see
    \begin{equation*}
        k(r_{p}(x))r_{p}'(x)
        =
        \frac{(a+b)(q - r_{p}(x))r_{p}'(x)}{r_{p}(x)(1-r_{p}(x))}
        =
        \frac{(a+b)(q - r_{p}(x))(p-x)}{(p - r_{p}(x))x(1-x)}.
    \end{equation*}
    Rewriting
    \begin{equation*}
        k(x)
        =
        \frac{(a+b)(q-x)}{x(1-x)}
        =
        \frac{(a+b)(p-r_{p}(x))(q-x)}{(p-r_{p}(x))x(1-x)}
    \end{equation*}
    we see that
    \begin{equation*}
    \begin{split}
        &(p-r_{p}(x))x(1-x)h'(x)
        \\&\quad=
        (a+b)((q - r_{p}(x))(p-x) - (p-r_{p}(x))(q-x))h(x)
    \end{split}
    \end{equation*}
    so that it remains only to realise that \((q - r_{p}(x))(p-x) - (p-r_{p}(x))(q-x) = (q-p)(r_{p}(x) - x)\).
\end{proof}

Once again, this lemma allows us to directly read off monotonicity properties.

\begin{lemma} \label{lemma:hmono} For \(p \in (0, 1)\) let \(r_{p} \colon [0, 1]
  \to [0, 1]\) be the function defined in Lemma~\ref{lemma:rprops} and let
  \(g(x) = cx^{a}(1-x)^{b}\) for \(x \in (0, 1)\) and \(a, b, c \in \R\) such
  that \(c \neq 0\). Then
    \begin{enumerate}
    \item \label{lemma:hmono1} \(\frac{g(r_{p}(x))}{g(x)}\) is non-increasing if
      \(a \geq 0 \geq b\) (decreasing if either \(a > 0\) or \(0 > b\)),
    \item \label{lemma:hmono2} \(\frac{g(r_{p}(x))}{g(x)}\) is non-decreasing if
      \(a \leq 0 \leq b\) (increasing if either \(a < 0\) or \(0 < b\)),
    \item \label{lemma:hmono3} \(\frac{g(r_{p}(x))}{g(x)}\) is decreasing
      \((a+b)(a/(a+b) - p) > 0\),
    \item \label{lemma:hmono4} \(\frac{g(r_{p}(x))}{g(x)}\) is increasing
      \((a+b)(a/(a+b) - p) < 0\).
    \end{enumerate}
\end{lemma}
\begin{proof}
  Recalling that \(r_{p}\) is decreasing and observing that \(g\) is
  non-decreasing if \(a \geq 0 \geq b\) (increasing if either inequality is
  strict) and non-increasing if \(a \leq 0 \leq b\) (decreasing if either
  inequality is strict) we see that
  statements~\ref{lemma:hmono1}~and~\ref{lemma:hmono2} follow immediately.

    For the remaining statements Lemma~\ref{lemma:hdeq} tells us that
    \begin{equation*}
        (p-r_{p}(x))x(1-x)h'(x)
        =
        (a+b)(q - p)(r_{p}(x) - x)h(x)
    \end{equation*}
    where \(q = a/(a+b)\). Since \(h(x), x, (1-x) > 0\) and \(p-r_{p}(x)\) has
    the opposite sign to \(r_{p}(x) - x\) we see that \(h'(x)\) has the opposite
    sign to \((a+b)(q-p)\).
\end{proof}

Lemma~\ref{lemma:rmono}~and~\ref{lemma:hmono} give us a number of conditions
under which \(-r_{p}(x)g(r_{p}(x))/g(x)\) is monotone. Namely, when \(p \leq
1/2\) and either \(a \geq 0 \geq b\) or \((a+b)(a/(a+b) - p) > 0\) or when \(p
\geq 1/2\) and either \(a \leq 0 \leq b\) or \((a+b)(a/(a+b) - p) < 0\). These
are simply the cases where \(-r_{p}(x)\) and \(g(r_{p}(x))/g(x)\) are both
non-increasing or both non-decreasing. One might be tempted to try to deduce something
for the ``mixed'' cases by directly differentiating the product
\(-r_{p}(x)g(r_{p}(x))/g(x)\) and then using
Lemma~\ref{lemma:rprops}~and~\ref{lemma:hdeq} similar to the proofs of
Lemma~\ref{lemma:rmono}~and~\ref{lemma:hmono}. Sadly most or all of these cases
appear to be genuinely non-monotonic and other means must be found to generalise
the result.

The following proof is now simple.
\begin{proof}[Proof of Theorem~\ref{thm:betamono}]
  Let \(a, b\), and \(p\) be as in the statement and consider \(m > n >
  -\min(a/p, b/(1-p))\). Define \(W_{n}\) as in the statement and \(W_{m}\)
  analogously.

  Define \(r = r_{p}\) as in Lemma~\ref{lemma:rprops} and let \(f(x) =
  x^{p}(1-x)^{1-p}\). Write the (Lebesgue) densities of the distributions of
  \(W_{n}\) and \(W_{m}\) as \(f_{n}(x) = c_{n}f(x)^{n}g(x)\) and \(f_{m}(x) =
  c_{m}f(x)^{m}g(x)\), respectively, where \(g(x) = x^{a-1}(1-x)^{b-1}\) and
  \(c_{n}, c_{m}\) are appropriate normalising constants.

  By Lemma~\ref{lemma:powerdist3} we have for all \(t\) that \(\prob(W_{m} > t
  \mid W_{m} > p) \leq \prob(W_{n} > t \mid W_{n} > p)\). In particular, it
  follows that \(\expect(h(W_{m}) \mid W_{m} > p) \leq \expect(h(W_{n}) \mid
  W_{n} > p)\) for any non-decreasing \(h\), with the inequality reversed for
  any non-increasing \(h\). Consider the case \(h(x) = -r'(x)g(r(x))/g(x)\). We
  have by Lemma~\ref{lemma:oddsexpect} that \(\prob(W_{m} \leq p) \leq
  \prob(W_{n} \leq p)\) if \(h\) is non-decreasing and the reverse inequality if
  \(h\) is non-increasing.

  To conclude the proof, we simply consider combinations of values of \(p\),
  \(a\), and \(b\) such that \(-r'(x)\) and \(g(r(x))/g(x)\) are either both
  non-increasing or both non-decreasing using
  Lemma~\ref{lemma:rmono}~and~Lemma~\ref{lemma:hmono}. If \(p \leq 1/2\), then
  \(r\) is convex, so that \(-r'\) is non-increasing. For \(g(r(x))/g(x)\) to
  also be non-increasing we need either \(a \geq 1 \geq b\) or
  \((a+b-2)((a-1)/(a+b-2) - p) > 0\). Similarly, for the case \(p \geq 1/2\) we
  need \(b \geq 1 \geq a\) or \((a+b-2)((a-1)/(a+b-2) - p) < 0\).
\end{proof}

\printbibliography

\end{document}